\documentclass{amsart}

\usepackage[utf8]{inputenc}
\usepackage[T1]{fontenc}

\usepackage[hidelinks]{hyperref}
\hypersetup{colorlinks=true,citecolor=blue,linkcolor=blue,
	filecolor=blue,urlcolor=blue}

\usepackage[capitalize]{cleveref}

\usepackage{amssymb}
\usepackage{comment}
\usepackage{color}

\newtheorem{theorem}{Theorem}[section]
\newtheorem{lemma}[theorem]{Lemma}

\newtheorem{proposition}[theorem]{Proposition}

\theoremstyle{definition}

\theoremstyle{remark}

\numberwithin{equation}{section}

\renewcommand{\le}{\leqslant}
\renewcommand{\ge}{\geqslant}
\renewcommand{\leq}{\leqslant}

\renewcommand{\phi}{\varphi}

\begin{document}

\title{Differential-Difference Properties of Hypergeometric Series}

\author[N. Brisebarre]{Nicolas Brisebarre}
\address{Université de Lyon, CNRS, ENS de Lyon, Inria, Université
    Claude-Bernard Lyon~1, Laboratoire LIP (UMR 5668), Lyon, France.}
\curraddr{}
\email{Nicolas.Brisebarre@ens-lyon.fr}
\thanks{This work was partly supported by the NuSCAP ANR-20-CE48-0014 project of the French \emph{Agence Nationale de la Recherche}.}

\author[B. Salvy]{Bruno Salvy}
\address{Université de Lyon, CNRS, ENS de Lyon, Inria, Université
    Claude-Bernard Lyon~1, Laboratoire LIP (UMR 5668), Lyon, France.}
\curraddr{}
\email{Bruno.Salvy@ens-lyon.fr}

\begin{abstract}
Six families of generalized hypergeometric series in a variable~$x$ and an arbitrary number of parameters are considered. Each of them is indexed by an integer~$n$. Linear recurrence relations in~$n$ relate these functions and their product by the variable~$x$. We give explicit factorizations of these equations as products of first order recurrence operators. Related recurrences are also derived for the derivative with respect to~$x$. These formulas generalize well-known properties of the classical orthogonal
polynomials.
\end{abstract}



\maketitle

\section{Introduction}
Among many identities, the classical families of orthogonal polynomials satisfy
a three-term recurrence
\begin{equation}\label{eq:3-term}
xp_n(x)=a_np_{n+1}(x)+b_np_n(x)+c_np_{n-1}(x),
\end{equation} 
where $a_n,b_n,c_n$ are rational functions of~$n$~\cite{OlverLozierBoisvertClark2010};
a differential-difference equation
\begin{equation}\label{eq:mixed}
\pi(x)p_n'(x)=\alpha_np_{n+1}+\beta_np_n+\gamma_np_{n-1},
\end{equation} 
with $\pi(x)$ a polynomial of degree at most~2 and
$\alpha_n,\beta_n,\gamma_n$ rational functions 
of~$n$~\cite{Lewanowicz2002};
they also satisfy a second order linear differential equation with
polynomial
coefficients~\cite{Szego1975}.

Furthermore, the classical families of Jacobi, Laguerre, Hermite and Bessel are
\emph{hypergeometric}. They can be expressed
in terms of the functions ${}_2F_1$ and ${}_2F_0$~\cite[\S18.5(iii),\S18.34]
{OlverLozierBoisvertClark2010}, special cases of the generalized hypergeometric
series (see \cref{sec:def_hgm} for definitions).

In this work, we focus on two properties of the classical orthogonal polynomials
that persist for some of the generalized hypergeometric series. The
first one is a
generalization of
the 3-term recurrence~\eqref{eq:3-term}: 
\begin{equation}\label{eq:X}
x(A_m(n)p_{n+m}(x)+\dots+A_0(n)p_n(x))=B_\ell(n)p_{n+\ell}
(x)+\dots+B_0(n)p_n(x),
\tag{$\mathcal X$}
\end{equation}
where the $A_i$ and $B_j$ are polynomials in~$n$, but do not depend on~$x$.
In other words, the sequence $(p_n)$ satisfies a linear recurrence with
coefficients that are polynomials of degree
at most~1 in~$x$. 

The second relation is a
differential-difference equation
of the form
\begin{equation}\label{eq:D}
\left(C_r(n)p_{n+r}(x)+\dots+C_0(n)p_n(x)\right)'=
D_s(n)p_{n+s}(x)+\dots+D_0(n)p_n(x),\tag{$\mathcal D$}
\end{equation}
where again $C_i$ and $D_j$ are polynomials in~$n$, but do not depend on~$x$. 

We did not find any explicit mention of the relation~\eqref{eq:D} for
classical orthogonal polynomials in the literature. It can be seen by
differentiating
the three-term
recurrence \cref{eq:3-term}, which allows one to rewrite $xp_n'$ as a linear combination of shifts of~$p_n$ and $p_n'$ with coefficients that do not depend on~$x$. Then by induction, for all nonnegative integers~$k$, $x^kp_n$ and $x^kp_n'$
can also be written that way. 

Another equation of interest is a mixed
difference-differential equation
\begin{equation}\label{eq:M}
\pi(x)p_n'(x)=E_{-s}(n)p_{n-s}(x)+\dots+E_{t}(n)p_{n+t}(x),
\tag{$\mathcal M$}
\end{equation}
with $\pi(x)$ a polynomial in~$x$ and $E_{-s}(n),\dots,E_{t}(n)$
rational functions
of~$n$ that do not depend on~$x$. When the $p_n$ are orthogonal
polynomials, such relations characterize  semi-classical polynomials~\cite{Maroni1987,Maroni1991}.
A generalization of the derivation above is given in 
\cref{prop:XM-give-XD}. It shows that any solution of both an
equation of type~\eqref{eq:X} and an equation of type~\eqref{eq:M} also
satisfies an equation of type~\eqref{eq:D}.
Actually, \cref{eq:D} is strictly more general, as there are 
functions satisfying equations of type~\eqref{eq:X}
and~\eqref{eq:D} that do not satisfy any equation of 
type~\eqref{eq:M}. (An example is given in \cref{subsubsec:DmgM}.)

\subsection*{Contribution} 
Understanding the extent to which relations~\eqref{eq:X}
and~\eqref{eq:D} exist for
hypergeometric polynomials was our initial motivation for this
study.
We focus on six families of
generalized
hypergeometric series, given in \cref{table:families}. (Basic definitions and
properties are recalled in \cref{sec:def_hgm}.)
We show that they satisfy equations of the
types~\eqref{eq:X} and~\eqref{eq:D} and we give explicit
factorizations of the
linear recurrences that appear in these equations, as products of first
order operators.
\begin{table}[t]
\begin{small}
\begin{align}
{}_{p}F_{m+q}&\!\left(\left.\begin{matrix}
(a_p)\\
\Delta(m;\lambda+n),(b_q)
\end{matrix}\right|x\right),\tag{I}\\
{}_{m+p}F_{q}&\!\left(\left.\begin{matrix}
\Delta(m;1-\lambda-n),(a_p)\\
(b_q)
\end{matrix}\right|x\right),\tag{II}\\
{}_{m+p}F_{m+q}&\!\left(\left.\begin{matrix}
\Delta(m;\mu+n),(a_p)\\
\Delta(m;\lambda+n),(b_q)
\end{matrix}\right|x\right),\tag{III}\\
{}_{m+p}F_{m+q}&\!\left(\left.\begin{matrix}
\Delta(m;1-\lambda-n),(a_p)\\
\Delta(m;1-\mu-n),(b_q)
\end{matrix}\right|x\right),\tag{IV}\\
{}_{2m+p}F_{q}&\!\left(\left.\begin{matrix}
\Delta(m;1-\lambda-n),\Delta(m;\mu+n),(a_p)\\
(b_q)
\end{matrix}\right|x\right),\tag{V}\\
{}_{p}F_{2m+q}&\!\left(\left.\begin{matrix}
(a_p)\\
\Delta(m;\lambda+n),\Delta(m;1-\mu-n),(b_q)
\end{matrix}\right|x\right).\tag{VI}
\end{align}
\end{small}
\caption{Families of generalized hypergeometric functions. There, $m$ denotes a
positive integer; $\lambda\neq\mu$ are two arbitrary, but distinct,
constants; $n$ is a nonnegative integer; $(a_p)$ denotes the sequence
of complex numbers
$a_1,\dots,a_p$ and likewise for~$(b_q)$; $\Delta(m;x)$ is the
sequence $x/m,(x+1)/m,\dots,(x+m-1)/m$.}
\label{table:families}
\end{table}

Family~I generalizes the Bessel functions
\[J_\nu(z)=\frac{\left(\frac z2\right)^\nu}{\Gamma(\nu+1)}{}_0F_1\left
(\left.\begin{matrix}\\ \nu+1\end{matrix}\right|-\frac{z^2}4\right)\]
up to a monomial factor.

Families II and V generalize the extended Laguerre and Jacobi
functions
\begin{align}\label{eq:exLagJac}
&
{}_{p+1}F_{q}\!
\left(
\left.
\begin{matrix}
-n-\lambda,a_1,\dots,a_p\\
b_1,\dots,b_q
\end{matrix}
\right|
x
\right),
&
{}_{p+2}F_{q}\!\left(\left.\begin{matrix}
-n-\lambda,n+\mu,a_1,\dots,a_p\\
b_1,\dots,b_q
\end{matrix}\right|x\right).
\end{align}
The classical Laguerre and Jacobi polynomials are recovered, up to a simple
change of
variable, when
$\lambda=0,p=0,q=1$.

The case $\lambda=0$ of Family~II was used by
Brafman~\cite{Brafman1957} to produce several generating functions.
Other generating functions are known for cases of 
families~II, IV, V that correspond to polynomials~\cite[chap.~7]{SrivastavaManocha1984}.

\subsection*{Related work}
The closest relation to our work that we could find in the literature are
explicit
equations of the types~\eqref{eq:D} and~\eqref{eq:X} given by Fields, Luke
and Wimp for the extended Laguerre
and Jacobi functions of \cref{eq:exLagJac}
in their study of Pad\'e approximants of hypergeometric functions~\cite{FieldsLukeWimp1968,Wimp1975,Luke1975}. 
In the
case of Jacobi functions, this was later improved by Lewanowicz~
\cite{Lewanowicz1985}.

We show in \cref{sec:examples} how the results of Fields, Luke and Wimp are
recovered as special cases of our main result.

\section{Definitions and Notation}\label{sec:def-and-not}
\subsection{Generalized Hypergeometric Function}\label{sec:def_hgm}
We gather here basic material on \emph{generalized hypergeometric series} (see \cite[chap.16]
{OlverLozierBoisvertClark2010}).
Those are classically
defined as the formal power series
\begin{equation}\label{eq:defhgm}
{}_pF_q\left(\left.\begin{matrix}a_1,\dots,a_p\\ b_1,\dots,b_q
\end{matrix}\right|x\right):=\sum_{k\ge0}\frac{(a_1)_k\dotsm(a_p)_k}{
(b_1)_k\dotsm
(b_q)_k}\frac{x^k}{k!},
\end{equation}
where $(a)_k=a(a+1)\dotsm(a+k-1)$ is the Pochhammer symbol and it is assumed
that none of the $b_i$ is a nonpositive integer, so that the denominators are
nonzero for all~$k$. 
In the hypergeometric representations of the classical
orthogonal polynomials, the parameter $a_1$ is $-n$, so that the series
becomes a polynomial of degree~$n$. In general, the power series ${}_pF_q$ is a
polynomial if one of
the~$a_i$ is a negative integer. Except in this situation, it is
divergent in the neighborhood of~0 when $p>q+1$, and convergent
otherwise.

If $U_k$ denotes the coefficient of~$x^k$ in the power 
series~\eqref{eq:defhgm}, then the definition of the
Pochhammer symbol implies
\begin{equation}\label{eq:rec_hgm}
(k+1)\prod_{j=1}^q(b_j+k)U_{k+1}=\prod_{j=1}^p(a_j+p)U_k.
\end{equation}
It follows that the
generalized hypergeometric series satisfies a linear differential equation~\cite
[16.8.3]{OlverLozierBoisvertClark2010}
\begin{equation}\label{deq:hgm}
\left(\vartheta(\vartheta+b_{1}-1)\cdots(\vartheta+b_{q}-1)-z(\vartheta+a_{1})%
\cdots(\vartheta+a_{p})\right)w=0,
\end{equation}
where $\vartheta=z\frac{d}{dz}$ and whose order, $\max(p,q+1)$, is minimal.

\subsection{Linear Recurrence Operators}\label{sec:linrecop}
Our results and proofs are more conveniently expressed using linear recurrence
operators. We denote by $S_n$ the \emph{shift operator} that maps a sequence $
(u_n)$ to the sequence $(u_{n+1})$. The linear recurrence operators we consider
are polynomials in~$S_n$ with coefficients that are polynomial or rational
functions in~$n$. Their addition is that of commutative polynomials and
their multiplication follows from the commutation rule $S_na(n)=a(n+1)S_n$, for
any rational function~$a(n)$. These polynomials admit a Euclidean division on
the right, and a Euclidean algorithm that allows for the definition of greatest
common right divisors and least common left multiples (denoted
lclm)~\cite{Ore1933}. They are both defined up to a rational
factor. These
notions are exemplified in the following two results that are used
in the next section.
\begin{lemma}For distinct $a_1,\dots,a_p$ that do not depend on~$n$,
\begin{multline*}
\operatorname{lclm}\!\left(\frac{n+a_1+1}{n+1}S_n-1,\dots,\frac{n+a_p+1}
{n+1}S_n-1\right)\\ 
=\left(\frac{n+a_p+p}{n+1}S_n-1\right)\dotsm\left(
\frac{n+a_1+1}{n+1}S_n-1\right)=:Q_p.
\end{multline*}
\end{lemma}
\begin{proof}
Let $M_p$ denote the lclm and $Q_p$ the product. 
We first show that the remainder
of the right division of~$Q_p$ by 
\[L_b:=S_n-\frac{n+1}{n+b+1}\]
is
\[R_{p,b}(n)=\frac{(a_1-b)\dotsm(a_p-b)}{(n+b+1)\dotsm(n+b+p)}.\]
For $p=0$, the remainder of the division of~$Q_0=1$
by~$L_b$ is~1, which corresponds to the empty product.
Next, by induction, since $Q_p\bmod L_b=R_{p,b}(n)$, it is sufficient to
consider
\begin{align*}
& \left(\frac{n+a_{p+1}+p+1}{n+1}S_n-1\right)R_{p,b}(n)\\
&\quad=
\frac{n+a_{p+1}+p+1}{n+1}R_{p,b}(n+1)S_n-R_{p,b}(n)\\
&\quad=
\frac{n+a_{p+1}+p+1}{n+1}R_{p,b}(n+1)L_b+\left(\frac{n+a_{p+1}+p+1}
{n+b+1}R_{p,b}(n+1)-R_{p,b}(n)\right).
\end{align*}
The remainder of the right division by~$L_b$ factors as
\[R_{p,b}(n)\left(\frac{n+a_{p+1}+p+1}
{n+b+p+1}-1\right)=R_{p,b}(n)\frac{a_{p+1}-b}{n+b+p+1}=R_{p+1,b},\]
which concludes the induction.
Taking~$b=a_i$ for $i=1,\dots,p$ makes $R_{p,b}=0$, showing that $M_p$
divides~$Q_p$ for all~$p$. 

Another
induction on~$p$ establishes that $M_p=Q_p$. The case
$p=1$ is clear.
If $a_{p+1}$ is distinct
from the other~$a_i$, taking
$b=a_{p+1}$ shows that the remainder~$R_{p+1,a_{p+1}}$ is not~0,
which implies that $\deg M_{p+1}=\deg M_p+1=\deg Q_{p+1}$, 
concluding the proof since $M_
{p+1}$ divides $Q_{p+1}$.
\end{proof}
\begin{lemma}\label{lemma:explicit-coeffs}
For arbitrary $a_1,\dots,a_p$ that do
not depend on~$n$, let
$Q_p$ be
the product defined in the previous lemma. Then, the coefficients of its
expansion in powers of~$S_n$ are given by
\[Q_p=\sum_{m=0}^p{\frac{c_m(n)}{(p-m)!}\frac{\prod_{i=1}^p(n+a_i+m)}{\prod_
{i=1}^m
{
(n+i)}}S_n^m},\]
with 
\[c_m(n)={}_{p+1}F_p
\!\left(\left.
\begin{matrix}m-p,n+m+a_1+1,\dots,n+m+a_p+1\\
n+m+a_1,\dots,n+m+a_p\end{matrix}\right|1\right).\]
\end{lemma}
\begin{proof}
The proof is by induction. For $p=0$, the empty product $Q_0$ is equal to~1 and
so is the
hypergeometric series, since $0$ is its top parameter. 
Next, multiplying the sum for~$Q_p$ by 
\[\frac{n+a_{p+1}+p+1}{n+1}S_n-1\] shows that the coefficient of~$S_n^m$
in~$Q_{p+1}$ is
\begin{gather*}
\frac{n+a_{p+1}+p+1}{n+1}\frac{c_{m-1}(n+1)}{(p-m+1)!}\frac{\prod_{i=1}^p
(n+1+a_i+m-1)}{\prod_{i=1}^{m-1}{(n+1+i)}}
-\frac{c_m(n)}{(p-m)!}\frac{\prod_{i=1}^p(n+a_i+m)}{\prod_{i=1}^m
{
(n+i)}}\\
=\frac{\prod_{i=1}^{p+1}{(n+a_i+m)}}{(p+1-m)!\prod_{i=1}^m(n+i)}
\frac{(n+a_{p+1}+p+1)c_{m-1}(n+1)-(p+1-m)c_m(n)}{n+a_{p+1}+m}.
\end{gather*}
The coefficient~$c_{k,m,n,p}$ of $x^k$ in the hypergeometric series
defining~$c_m(n)$ is
\[(m-p)_k\frac{(n+m+a_1+k)\dotsm(n+m+a_p+k)}{(n+m+a_1)\dotsm
(n+m+a_p)},\qquad 0\le k\le m-p\]
and~0 for $k>m-p$. From there, 
\begin{align*}
&\frac{(n+a_{p+1}+p+1)c_{k,m-1,n+1,p}-(p+1-m)c_{k,m,n,p}}{n+a_{p+1}+m}\\
&\qquad=\frac{c_{k,m,n,p}(m-p-1)}{n+a_{p+1}+m}\left(\frac{n+a_{p+1}+p+1}
{m-p+k-1}+1\right)\\
&\qquad=c_{k,m,n,p}\frac{m-p-1}{m-p+k-1}\frac{n+m+a_{p+1}+k}{n+m+a_
{p+1}}=c_{k,m,n,p+1}.\qedhere
\end{align*}
\end{proof}

\section{Main result}
\begin{theorem}\label{thm:main}
Let $F_n$ be in any of the six families
in \cref{table:families}, with parameters such that~$F_n$ is a
well-defined hypergeometric series for all~$n\in\mathbb N$. 
Then the power series
$xF_n$,
$F_n'$ and $F_n$ are related by the following two recurrence equations
\[\begin{split}
&\left((\epsilon_1\epsilon_2S_n)^{\theta}
\mathcal L_{p}^{q+1-2m\chi\epsilon_1}
\mathcal A\right)(xF_n)=\\
&\qquad\left(
m^{m(\epsilon\epsilon_1-\epsilon_2)}\left(
\frac{(\lambda+n)_m}{
(\mu+n)_m^\epsilon}\right)^{\!\epsilon_1}
(\epsilon_1\epsilon_2S_n)^{m-\theta}
\mathcal L_{q+1}^{p+2m\chi\epsilon_1}
\mathcal F_{0,q+1}\mathcal B\right)
(F_n),\\
&\left(
(\epsilon_1\epsilon_2S_n)^{m-\theta}
\mathcal L_{q}^{p+2m\chi\epsilon_1}
\mathcal B\right)(F_n')=\\
&\qquad
\left(
m^{m(\epsilon_2-\epsilon\epsilon_1)}\left(
\frac{(\lambda+n)_m}{
(\mu+n)_m^\epsilon}\right)^{\!-\epsilon_1}
\!(\epsilon_1\epsilon_2S_n)^{\theta}
\mathcal L_{p}^{q-2m\chi\epsilon_1}
\mathcal A\right)(F_n),
\end{split}\]
where $\epsilon=1$ if $F_n$
depends on~$\mu$ and $\epsilon=0$ otherwise; $\epsilon_1$ and
$\epsilon_2$ in $\{-1,1\}$ are
the signs in front of~$\mu$ and~$\lambda$ in the parameters of the
hypergeometric series
(for convenience, we let
$\epsilon_1=1$ when $\epsilon=0$), and
\begin{gather*}
a(n,k)=\frac{n+\lambda}
{n+\epsilon\mu}\,
\frac{n+\epsilon(\mu+\epsilon_1mk)}
{n+\lambda+\epsilon_2mk},\qquad\alpha(n,k)=\frac1
{n+\lambda+\epsilon_2mk},\\
C_i(n)=
\frac{n+\lambda}{n+\epsilon\mu}
(n+\epsilon\mu-\epsilon\epsilon_1\epsilon_2(\lambda+n+i)),
\quad \mathcal C_i=\frac1{C_i(n)}\left(S_n-\frac{\epsilon\epsilon_1
(n+\lambda)}{\epsilon_2
(n+\epsilon\mu)}\right),
\\
B_i(n,c)=-m\epsilon_2\alpha
(n+i-1,-c)C_{i-1}(n),\\
\mathcal F_{c,i}=\frac1{B_i(n,c)}\left(S_n-\frac{a(n,-c)\alpha
(n+i-1,-c)}
{\alpha(n,-c)}\right),\\
\mathcal B=\mathcal F_{b_q-1,q}\dotsm\mathcal F_
{b_1-1,1},\quad 
\mathcal A=\mathcal F_{a_p,p}\dotsm\mathcal F_{a_1,1},
\quad 
\mathcal L_i^{j}=\begin{cases}\mathcal C_{j-1}\dotsm\mathcal
C_i&\text{if $i<j$},\\ 1&\text{otherwise},\end{cases}\\
\theta=\begin{cases}m&\quad\text{if
$\epsilon_1=1$,}\\0&\quad\text{otherwise,}\end{cases}
\quad
\chi=\begin{cases}0&\quad\text{if $\epsilon_1=\epsilon_2$,}\\
1&\quad\text{otherwise.}
\end{cases}
\end{gather*}
In the special case when~$\epsilon = 1$ and $\epsilon_1=\epsilon_2$ (families III
and~IV) and furthermore~$\mu-\lambda\in\{0,\dots,\max(p-1,q)\}$, these
formulas do not apply directly as a required~$C_i(n)$ vanishes. An
identity is recovered by truncating the operators, keeping all the
right factors up to the division by this~$C_i(n)$ excluded. In that
situation, both terms of the identity vanish.
\end{theorem}
With extra care, the cases of parameters that make the hypergeometric
series be well defined only for sufficiently large or sufficiently
small~$n$ can be handled as well.

\section{Examples}\label{sec:examples}
\subsection{Legendre Polynomials}

One of the simplest examples, the Legendre polynomials, may help clarify
the notation. These polynomials are given as 
\[P_{n}\left(x\right)={{}_{2%
}F_{1}}\left(\left.{-n,n+1\atop1}\right|\frac{1-x}{2}\right).\]
Thus they equal $F_n((1-x)/2)$, where $F_n$ is the special case of
family~(V) with~$p=0$, $m=q=b_1=\lambda=\mu=1$. With these values of the parameters, the theorem gives~$\epsilon=1,\epsilon_1=1,\epsilon_2=-1,\theta=\chi=1$,
\begin{gather*}
a(n,0)=1,\quad\alpha(n,0)=\frac{1}{n+1},\quad
B_1(n,0)=2,\quad
B_2(n,0)=\frac{2n+3}{n+2},
\\
\mathcal B=\mathcal F_{0,1}=\frac1{2}(S_n-1),\quad\mathcal
A=\mathcal L_2^1=\mathcal L_2^2=1,\quad
\mathcal F_{0,2}=\frac{n+2}{2n+3}\left(S_n-\frac{n+1}{n+2}\right),\\ 
\mathcal L_1^2=\mathcal C_1=\frac{1}{2n+3}(S_n+1).\\ 
\intertext{The relations given by the theorem are therefore}
\begin{cases}-S_n (xF_n)=\frac{n+2}{2n+3}\left(S_n-\frac{n+1}
{n+2}\right)\frac1
{2}(S_n-1)(F_n), \\ 
\frac{1}{2(2n+3)}(S_n+1)(S_n-1)(F_n')=-S_n(F_n).
\end{cases}
\end{gather*}
Replacing $x$ by $(1-x)/2$ in the first equation and using the fact
that $F_n'((1-x)/2)=-2P_n'(x)$ in the second one retrieves
relations that can be seen to be equivalent to classical ones~\cite
[18.9.1,18.9.17]{OlverLozierBoisvertClark2010}
\[
\begin{cases}\frac{x-1}2P_{n+1}=\frac{1}{2(2n+3)}((n+2)P_{n+2}-
(2n+3)P_{n+1}+
(n+1)P_n),\\
\frac1{2n+3}(P_{n+2}'-P_n')=P_{n+1}.
\end{cases}
\]

\subsection{Extended Laguerre Polynomials}
These are the special case of family (II) with $m=1$ and $\lambda=1$.
Then $\epsilon=0,\epsilon_1=1,\epsilon_2=-1$ and $\theta=\chi=1$.
We give the explicit factorization of the relations given by 
\cref{thm:main}. The first one can be compared with the formula
given by Fields, Luke and Wimp~\cite[Cor.~2.2]{FieldsLukeWimp1968}\footnote{Fields, Luke and Wimp take one of the $a_i$s to be~1, but this does
not impact the results mentioned here.}.

The quantities involved in the theorem are
\[\mathcal C_i=\frac{1}{n+1}S_n,\quad
\mathcal F_{c,i}=\frac{n+i+c}{n+1}S_n-1,\]
whence the formulas
\begin{multline*}
-\left(\frac{S_n}{n+1}\right)^{\max(0,q-p-1)}S_n\prod_
{i=1}^p\left(\frac{n+a_i+i}{n+1}S_n-1\right)(xF_n)=\\
(n+1) \left(\frac{S_n}{n+1}\right)^{\max(0,p+1-q)}\left(\frac{n+q+1}
{n+1}S_n-1\right)\prod_
{i=1}^q\left(\frac{n+b_i+i-1}{n+1}S_n-1\right)(F_n),\\ 
(n+1)\left(\frac{S_n}{n+1}\right)^{\max(0,p+2-q)}\prod_
{i=1}^q\left(\frac{n+b_i+i-1}{n+1}S_n-1\right)(F_n')=\\ 
\left(\frac{S_n}{n+1}\right)^{\max(0,q-p-2)}S_n\prod_
{i=1}^p\left(\frac{n+a_i+i}{n+1}S_n-1\right)(F_n),
\end{multline*}
where the products are to be interpreted as follows for
any~$M_i$:
\[\prod_{i=j}^kM_i=\begin{cases}M_k\dotsm M_j,\quad&\text{if $k\ge
j$,}\\ 1&\text{otherwise.}\end{cases}\]

The explicit hypergeometric coefficients given by Fields, Luke and
Wimp come from \cref{lemma:explicit-coeffs}. Note however that
they give a non-homogeneous equation, while our result is a homogeneous one.

\subsection{Special Case}
As an illustration of the special case at the end of the theorem,
consider
\[F_n={}_{2}F_{2}\!\left(\left.\begin{matrix}
\lambda+n+1,a\\
\lambda+n,b
\end{matrix}\right|x\right).\]
The second identity of the theorem does not lead to any division by~0
and gives
\begin{multline*}
\left(-\frac{(n+\lambda+1)(n+\lambda+1-b)}
{n+\lambda}S_n+n+\lambda+2-b\right)(F_n')=\\
\frac{n+\lambda+1}{n+\lambda}S_n\left(-\frac{(n+\lambda+1)(n+\lambda-a)}
{n+\lambda}S_n+n+\lambda+1-a
\right)(F_n).
\end{multline*}
The first one involves $\mathcal C_1$ and $\mathcal F_{0,2}$, both
with a division by $C_1(n)=0$. Stopping before this division gives an
identity where both sides are~0:
\begin{multline*}
\left(S_n-\frac{n+\lambda}{n+\lambda+1}\right)\left(-\frac{(n+\lambda+1)(n+\lambda-a)}
{n+\lambda}S_n+n+\lambda+1-a
\right)(xF_n)=\\
\left(S_n-\frac{n+\lambda}{n+\lambda+1}\right)
\left(-\frac{(n+\lambda+1)(n+\lambda+1-b)}
{n+\lambda}S_n+n+\lambda+2-b\right)(F_n)
=0.\end{multline*}

\section{Proof of the main result}
If~$F_n$ is any of the power series from~I to~VI and  $U_{n,k}$ denotes the
coefficient of~$x^k$ in~$F_n$, we compare the actions of certain recurrence operators independent
of~$k$ on the coefficient of~$x^k$ in~$xF_n$ and $F_n$ (for
$\mathcal X$),  resp. the coefficient of~$x^{k-1}$ in~$F_n'$ and~$F_n$ (for $\mathcal D$),
i.e.,
\begin{equation}\label{eq:very-simple}
U_{n,k-1} \textrm{ and } U_{n,k},\quad \textrm{ resp. }\quad kU_
{n,k}  \textrm{ and } U_{n,k-1}.
\end{equation}
The proof is constructive: iteratively, operators in the shift~$S_n$
\emph{with
coefficients that do not depend on~$k$} are built so that in the
end~$\mathcal X$ and $\mathcal D$ are obtained.

The following two identities are satisfied
by $U_{n,k}$
in all cases:
\begin{equation}
\frac{U_{n+1,k}}{U_{n,k}}
=a(n,k),\qquad
\frac{U_{n,k-1}}{U_{n,k}}=\psi_0(n,k)\frac{(b_1+k-1)\dotsm
(b_q+k-1)k}{(a_1+k-1)\dotsm(a_p+k-1)},\label{eq:shifts}
\end{equation}
with a $\psi_0$ that does not depend on~$(a_p)$ and~$(b_q)$ and with
$a(n,k)$ from the theorem.
Writing~$U_{n+1,k-1}$ as either 
$\left.U_{n,k-1}\right|_{n\mapsto n+1}$ or 
$\left.U_{n+1,k}\right|_{k\mapsto k-1}$ and using the relations from
\cref{eq:shifts} shows
that $\psi_0$
satisfies
the
relation
\begin{equation}\label{eq:psi}
\psi_0(n+1,k)a(n,k)=\psi_0(n,k)a(n,k-1).
\end{equation}
From there and the initial value~$\psi_0(0,k)$, it follows that
\begin{equation}\label{eq:defpsi_0}
    \psi_0(n,k)=(\epsilon_1\epsilon_2)^mm^{m (\epsilon\epsilon_1-\epsilon_2)}
 \frac{(n+\lambda+\epsilon_2km-\delta_{\epsilon_2,1}m)_{m}^{\epsilon_2}}{
(n+\mu+\epsilon_1 km-\delta_{\epsilon_1,1}m)_{m}^{\epsilon\epsilon_1}},   
\end{equation}
with $\delta$ the Kronecker symbol.
A key property satisfied by these 6~families and that makes our approach work 
is that \cref{eq:shifts} implies that for any~$c$
and~$i$,
\begin{equation}\label{eq:key}\frac{a(n,k)}{\alpha(n,k)}-\frac{a
(n,-c)\alpha(n+i-1,-c)}{\alpha
(n,-c)\alpha(n+i-1,k)}=(k+c){B_i(n,c)},\end{equation}
synthesizing a crucial factor~$k+c$, while $B_i(n,c)$ does not depend on~$k$
and therefore neither does $\mathcal F_{c,i}$. 

We first assume that $C_i(n)$ does not vanish for $i\in\mathbb N$, and
therefore that neither does $B_i(n,c)$
(the other case is addressed at the end of the proof.)
Introducing
\[
\phi_i(n,k)=\prod_{j=0}^{i-1}\alpha(n+j,k),\qquad
\psi_i(n,k)=\psi_0(n,k)\phi_i(n,k-1),
\]
a direct computation using~\cref{eq:shifts,eq:psi,eq:key} shows that
\begin{equation}\label{eq:fcc+1}
\begin{split}
\mathcal F_{c,i}(\phi_{i-1}(n,k)U_{n,k})&=(k+c)\phi_i
(n,k)U_{n,k},\\
\mathcal F_{c+1,i}(\psi_{i-1}(n,k)U_{n,k})&=(k+c)\psi_i
(n,k)U_{n,k}.
\end{split}
\end{equation}
It follows by induction that
\begin{equation}\label{eq:rec-A-B}
\begin{split}
\mathcal F_{0,q+1}\mathcal B(U_{n,k})&=\phi_{q+1}(n,k)(b_1+k-1)\dotsm(b_q+k-1)kU_
{n,k},\\
\mathcal A(U_{n,k-1})&=\quad\psi_p(n,k)(b_1+k-1)\dotsm(b_q+k-1)kU_{n,k},\\
\mathcal B(kU_{n,k})&=\quad\phi_{q}(n,k)(b_1+k-1)\dotsm(b_q+k-1)kU_{n,k}.
\end{split}
\end{equation}
The first two identities will be used to derive~$\mathcal X$;
$\mathcal D$ will come from the last two.

Note that $\phi_p$ and $\psi_q$ do not depend on the parameters $a_i$
and $b_i$ of the hypergeometric series. These operators thus allow for
reducing the computation to the case when $p=q=0$. The next step is to
increase the smaller of the indices~$q+1$ and~$p$
until a given target is reached. This is obtained with $\mathcal C_i$,
as a direct computation using \cref{eq:shifts,eq:psi} shows that it
increases indices:
\begin{equation}\label{eq:Ci}
  \mathcal C_i(\phi_i(n,k)U_{n,k})=\phi_{i+1}(n,k)U_
{n,k},\qquad
\mathcal C_i(\psi_i(n,k)U_{n,k})=\psi_{i+1}(n,k)U_{n,k}.
\end{equation}
By induction, it follows that
\begin{equation} \label{eq:Lij}
\begin{split}
  \mathcal L_i^j(\phi_i(n,k)U_{n,k})&=\phi_{\max(i,j)}(n,k)U_
{n,k},\\ 
\mathcal L_i^j(\psi_i(n,k)U_{n,k})&=\psi_{\max(i,j)}(n,k)U_{n,k}.
\end{split}
\end{equation}
Note that since $\mathcal C_i$ does not depend on~$k$, the same
identities hold with factors depending on~$k$ but not~$n$ on both
sides, such as the factors in \cref{eq:rec-A-B}.

The final step is to map a $\psi_i$ to a $\phi_j$ or the converse.
This is obtained by considering $\psi_0(n+m,k)U_{n+m,k}/U_{n,k}$:
using \cref{eq:defpsi_0,eq:shifts} and the definition of the
sequence~$a(n,k)$ gives
\[\psi_0(n+m,k)\frac{U_{n+m,k}}{U_{n,k}}=
(\epsilon_1\epsilon_2)^mm^{m(\epsilon\epsilon_1-\epsilon_2)}
L(k,m,\mu)M(k,m,\mu),\]
with 
\begin{align*}
L(k,m,\lambda)&=
\frac{
(n+\lambda)_{m}(n+\lambda+\epsilon_2km+(1-\delta_{\epsilon_2,1})m)_m
^{\epsilon_2}}{(n+\lambda+\epsilon_2mk)_m},\\ 
M(k,m,\mu)&=\frac{(n+\epsilon(\mu+\epsilon_1km))_m}{{
(n+\epsilon\mu)_m
(n+\mu+\epsilon_1 km+(1-\delta_{\epsilon_1,1})m)_m^
{\epsilon\epsilon_1}}}.
\end{align*}
Both products simplify depending on the values of
$\epsilon,\epsilon_1,\epsilon_2$, giving
\[
L(k,m,\lambda)={(n+\lambda)_m}\times\begin{cases}
1&\text{if $\epsilon_2=1$,}\\
\frac{1}
{(n+\lambda-mk)_{2m}}&\text{otherwise.}
\end{cases}
\]
\[M(k,m,\lambda)=
\frac1{(n+\mu)_m^\epsilon}\times\begin{cases}
1&\text{if $\epsilon_1=1$,}\\
(n+\mu-km)_{2m}&\text{otherwise.}
\end{cases}
\]
Finally, we also have
\[\phi_p(n+m,k-1)=\begin{cases}\phi_p(n,k)&\text{if $\epsilon_2=1$},\\
\frac1{(n+2m+\lambda-mk)_p}&\text{otherwise.}
\end{cases}\]
Combining these identities shows that, when $\epsilon_1=1$,
\begin{align*}
(\epsilon_1\epsilon_2S_n)^m&(\psi_p(n,k)U_{n,k})\\
&
=(\epsilon_1\epsilon_2)^m\psi_p(n+m,k)\frac{U_{n+m,k}}{U_{n,k}}U_
{n,k},\\
&=(\epsilon_1\epsilon_2)^m\psi_0(n+m,k)\frac{U_{n+m,k}}{U_
{n,k}}\phi_p(n+m,k-1)U_{n,k},\\
&=m^{m(\epsilon-\epsilon_2)}L(k,m,\mu)M
(k,m,\mu)\phi_p(n+m,k-1)U_{n,k},\\
&=m^{m(\epsilon-\epsilon_2)}\frac{(n+\lambda)_m}{(n+\mu)_m^\epsilon}
U_{n,k}\times
\begin{cases}\phi_p(n,k)&\text{if $\epsilon_2=1$,}\\
\frac1{(n+\lambda-mk)_{2m}(n+2m+\lambda-mk)_p}&\text{otherwise,}
\end{cases}\\
&=m^{m(\epsilon-\epsilon_2)}\frac{(n+\lambda)_m}{(n+\mu)_m^\epsilon}
\phi_{p+2m\chi}U_{n,k}.
\end{align*}
A similar derivation gives
\[(\epsilon_1\epsilon_2S_n)^m(\phi_q(n,k)U_{n,k})=
m^{m(\epsilon+\epsilon_2)}
\frac{(\lambda+n)_m}{(\mu+n)_m^\epsilon}\psi_{q+2m\chi}
(n,k)U_{n,k}\quad(\epsilon_1=-1).
\]

As a consequence, we obtain that
\[\begin{split}
&\left((\epsilon_1\epsilon_2S_n)^\theta\mathcal L_{p}^
{q+1-2m\chi\epsilon_1}\right)(\psi_p(n,k)U_{n,k})=\\
&\qquad\begin{cases}m^{m(\epsilon-\epsilon_2)}
\frac{(\lambda+n)_m}{
(\mu+n)_m^\epsilon}\phi_{\max(p+2m\chi,q+1)}(n,k)U_{n,k},\quad&
\text{if $\theta=m$,}\\
\psi_{\max(p,q+1+2m\chi)}(n,k)U_{n,k},\quad&\text{otherwise.}
\end{cases}\\
&\left((\epsilon_1\epsilon_2S_n)^{m-\theta}\mathcal L_{q+1}^
{p+2m\chi\epsilon_1}\right)(\phi_{q+1}(n,k)U_{n,k})=\\
&\qquad\begin{cases}
\phi_{\max(p+2m\chi,q+1)}(n,k)U_{n,k},&\quad\text{if $\theta=m$,}\\
m^{m(\epsilon+\epsilon_2)}\frac{(\lambda+n)_m}{
(\mu+n)_m^\epsilon}\psi_{\max(p,q+1+2m\chi)}(n,k)U_
{n,k},&\quad
\text{otherwise.}
\end{cases}
\end{split}\]
Together with \cref{eq:rec-A-B} and \cref{eq:very-simple}, this concludes the proof of
the first
recurrence. The second one is obtained with $q$ in the place of $q+1$.

We conclude with the case when $C_i(n)$ vanishes for some~$i\in\{0,\dots,\max(p-1,q)\}$, that are the values of~$i$ used in
the theorem. From the definition of~$C_i$, this occurs when $\epsilon
= 1$, $\epsilon_1 = \epsilon_2$ and $i_0 := \mu - \lambda \in \{ 0,
\ldots, \max( p-1, q ) \}$. Then 
  $\chi = 0, C_{i_0} =B_{i_0+1} = 0$ and $C_{i} B_{i+1} \neq 0$ for $i
  \neq i_0-1$. We focus on the proof of the first relation between
  $xF_n$ and $F_n$ (the proof of the other relation is similar).
There are two cases depending on whether $i_0\le\min(p-1,q)$ or not. 
Eqs.~\eqref{eq:fcc+1} hold as long as $i < i_0$. Let $\mathcal G_i:=
B_i \mathcal F_i $. 
If $i_0\le\min(p-1,q)$, at index
$i_0$, \cref{eq:fcc+1} can be replaced by:
\begin{align*}
\mathcal G_{c,i_0+1}(\phi_{i_0}(n,k)U_{n,k})&= B_{i_0+1}(n,c) (k+c)\phi_i (n,k)U_{n,k} = 0,\\
\mathcal G_{c+1,i_0+1}(\psi_{i_0}(n,k)U_{n,k})&= B_{i_0+1}(n,c+1)(k+c)\psi_i (n,k)U_{n,k}=0.
\end{align*}
Thus in that case, with the notation $b_
{q+1}:=1$, the first two equations of~\eqref{eq:rec-A-B} become
\begin{align}
\mathcal G_{b_{i_0+1}-1,i_0+1}\mathcal F_{b_{i_0}-1,i_0}\cdots \mathcal F_{b_{1}-1,1}  (U_{n,k})&= 0,  \notag \\ 
\mathcal G_{a_{i_0+1},i_0+1}\mathcal F_{a_{i_0},i_0}\cdots \mathcal F_{a_{1},1} (U_{n,k-1})&= 0, \label{eq:Ga}
\end{align}
proving the first relation.

In the other case, Eqs.~\eqref{eq:rec-A-B} hold for all
required values
of~$i$. Eqs.~\eqref{eq:Ci} hold for $i < i_0$. Let $\mathcal D_{i} :=
C_i \mathcal C_i$. At index $i_0$, \eqref{eq:Ci} can be replaced by:
\begin{align*}
  \mathcal D_{i_0} (\phi_{i_0}(n,k)U_{n,k}) & =C_{i_0} \phi_{i_0+1}(n,k)U_{n,k} = 0,\\
\mathcal D_{i_0} (\psi_{i_0} (n,k)U_{n,k}) & =C_{i_0} \psi_{i_0+1}(n,k)U_{n,k}=0.
\end{align*}
If $p-1 < i_0 \leq q$, \cref{eq:Ga} holds and
\cref{eq:Lij} can be replaced by: 
\[
  \mathcal D_{i_0} \mathcal C_{i_0-1} \cdots \mathcal C_{p}  (\psi_{p} (n,k)U_{n,k}) = 0,
\]
which, combined with the second equation of~\eqref{eq:rec-A-B}, yields
\[
  \mathcal D_{i_0} \mathcal C_{i_0-1} \cdots \mathcal C_{p}  \mathcal A U_{n,k-1} = 0,
\]
concluding the proof of this case. The case  $q < i_0 \leq p - 1$ is
similar.

\section{Existence of Further Relations}
\Cref{thm:main} gives equations of types~$\eqref{eq:D}$ 
and~\eqref{eq:X} for some hypergeometric series. To conclude, we first
prove that there exist hypergeometric families
that do not satisfy such equations.
 Mixed equations of type~\eqref{eq:M} are also of
interest. They are considered in \cref{subsec:mixed}.

\subsection{More hypergeometric families?}
Not all
hypergeometric families satisfy equations like~\eqref{eq:D} or~\eqref{eq:X}. For
instance,
if~$\Phi_n$ is defined by
\[\Phi_n(x)={}_2F_1\left(\left.\begin{matrix}n+1,n+1\\ 1
\end{matrix}\right|x\right)=\sum_{k\ge0}{\frac{((n+1)_k)^2}{k!^2}x^k}=\frac{\sum_{i=0}^n{\binom{n}{i}^2x^i}}{(1-x)^{2n+1}},\]
then computing the behaviour at~1 of both members of an
identity~\eqref{eq:X} of the form
\[\sum_{m=0}^tA_m\Phi_{n-m}(x)=x\sum_{m=0}^sB_m\Phi_{n-m}(x)\]
with at least one of~$A_0$ or $B_0$ being nonzero implies
\[\frac{A_0\binom{2n}{n}}{(1-x)^{2n+1}}+O\!\left(\frac1{(1-x)^{2n-1}}\right)
=\frac{B_0\binom{2n}{n}}{(1-x)^{2n+1}}
-\frac{B_0(n/2+1)\binom{2n}{n}}{(1-x)^{2n}}+O\!\left(\frac1{(1-x)^{2n-1}}\right),
\]
a contradiction.

The key property shared by the families of \cref{table:families} that leads
to the existence of the relations~\eqref{eq:D} and~\eqref{eq:X} is given in
\cref{eq:key}. It could be the case that a variant of this
identity allows for generalizing the approach to other hypergeometric families. We have not been able
to do so.

\subsection{Mixed Difference-Differential Equations}
\label{subsec:mixed}
We now consider the existence of equations of type \eqref{eq:M}.
We show that ``\eqref{eq:X} + \eqref{eq:D}'' is strictly more
general than
\eqref{eq:M} in two steps. First \cref{prop:XM-give-XD} shows an
inclusion.
 Next, cases where 
\cref{thm:main} applies and no equation of type~\eqref{eq:M} exists
are exhibited in 
\S\ref{subsubsec:DmgM}. Finally, there are cases where the theorem
applies and a mixed equation~\eqref{eq:M} exists. This seems to be
restricted to hypergeometric series with few parameters;
those we found are listed in \S\ref{subsub:Mforhyp}.
\subsubsection{``\texorpdfstring{\eqref{eq:X}}{(X)}+\texorpdfstring{\eqref{eq:M}}{(M)}'' implies 
  \texorpdfstring{\eqref{eq:D}}{(D)}}\label{eq:MtoD}
\begin{proposition}\label{prop:XM-give-XD} If a family of power series
$(\phi_n(x))_n$
satisfies
a relation of type~\eqref{eq:X} and a relation of type~\eqref{eq:M},
then it satisfies a relation of type~\eqref{eq:D}.
\end{proposition}
\begin{proof}
We give a simple proof in terms of fractions of Ore
polynomials \cite{Ore1933}. It can be turned into
an effective proof by expressing each of the steps using the
extended Euclidean algorithm.

\Cref{eq:X} amounts to a fraction 
\[\mathcal F_\mathcal X=(A_m(n)S_n^{m}+\dots+A_0(n))^{-1}(B_\ell
(n)S_n^\ell+\dots+B_0(n))\]
that maps the sequence~$(\phi_n(x))_n$ to the sequence $(x\phi_n
(x))_n$.
Evaluating the polynomial~$\pi$ at this fraction gives another
fraction~$\mathcal F_\pi=\pi(\mathcal F_\mathcal X)$ that maps~$
(\phi_n(x))_n$ to~$(\pi(x)\phi_n
(x))_n$. Similarly a fraction~$\mathcal F_{\pi'}$ maps~$(\phi_n(x))_n$ to~$(\pi'(x)\phi_n
(x))_n$. Letting $$\mathcal P=E_{-s}(n)S^{-s}+\dots+E_{t}(n)S^t$$
and differentiating $\pi(x)\phi_n(x)$ gives
\[
  (\mathcal F_\pi(\phi_n))'=(\pi(x)\phi_n)'=\pi'(x)\phi_n+\pi(x)\phi_n'=
  (\mathcal F_{\pi'}+\mathcal P)(\phi_n).
\]
Reducing to the same denominator on both sides yields a
relation
of type~\eqref{eq:D}.
\end{proof}

\subsubsection{\texorpdfstring{\eqref{eq:D}}{(D)}
is more general than \texorpdfstring{\eqref{eq:M}}{(M)}}
\label{subsubsec:DmgM}
The functions
\begin{multline*}
F_n(x)={}_{1}F_{2}\!\left(\left.\begin{matrix}
n+2\\
n,1
\end{matrix}\right|x\right)=\sum_{k\ge0}\frac{(n+k)(n+k+1)x^k}
{n(n+1)\,k!^2}\\
=
\left(1+\frac x{n(n+1)}\right)I_0(2\sqrt{x})+\frac{2n+1}{n(n+1)}
\sqrt{x}I_0'(2\sqrt{x}), \end{multline*}
where $I_0$ is a modified Bessel function, form a special case of
family~III (with $p=\lambda=0, m=q=b_1=1, \mu=2$). 
From the
derivative
\[
F_n'(x)=\frac2{n}I_0(2\sqrt{x})+\left(\frac1{
\sqrt{x}}+\frac{\sqrt x}{n(n+1)}\right)I_0'(2
  \sqrt{x}),
\]
it follows that the existence of a relation of type~\eqref{eq:M}
for~$F_n$ would imply the existence of a linear differential equation
of order~1 for~$I_0$ with polynomial coefficients. This is impossible,
for instance because~$I_0$ has an infinite number of complex zeros.

By contrast, as a further illustration of the general situation in
\cref{thm:main} with
$\epsilon=\epsilon_1=\epsilon_2=1$, this
function satisfies the following equation of type~\cref{eq:X}:
\[\frac12S_n\left(\frac{n+2}nS_n-1\right)^2(xF_n)=
(S_n-1)n(n/2+1)(S_n-1)(F_n).\]
This is readily checked: both operators map~$(2n+1)/(n(n+1))$ to~0;
the first one maps~1 to~$1/((n+1)(n+2))$ and~$1/(n(n+1))$ to~0;
the second one maps~1 to~0 and~$1/(n(n+1))$ to~$1/((n+1)(n+2))$.

The equation of type~\eqref{eq:D} provided by the theorem
is
\[-(n/2+1)(S_n-1)(F_n')=\frac{n+2}{2n}S_n\left(
\frac{n+2}nS_n-1\right)(F_n).\]
Again, this can be checked directly by observing that the left
operator maps $
(1,1/(n(n+1)),1/n)$ to $
(0,1/(n(n+1)),(n+2)/(2n(n+1)))$ while the right one maps $(1,1/
(n
(n+1)),(2n+1)/(n(n+1)))$ to $((n+2)/(n(n+1)),0,1/(n(n+1)))$.

\subsubsection{Equations of type~\texorpdfstring{\eqref{eq:M}}{
(M)}}\label{subsub:Mforhyp}
By \cref{thm:main}, families I~to~VI satisfy relations of type \eqref{eq:X}
and \eqref{eq:D}. However, apart from types~I and~II, a relation of
type 
\eqref{eq:M} 
exists only in special cases, for low numbers of parameters.
We now list such relations. In all cases, the proof reduces
to comparing the coefficients of~$x^k$ on both sides.

\smallskip
\paragraph{\em Families I and II}
These families have only one occurrence of the
parameter~$n$. The value~$\epsilon=0$ leads to
\[a(n,k)=\frac{n+\lambda}{n+\lambda+\epsilon_2mk}\] in the identity
$U_{n+1,k}/U_{n,k}=a(n,k)$. It follows that
\[\frac1{a(n-1,k)}-1=\frac{\epsilon_2mk}{n+\lambda-1}\]
which leads to the existence of a difference-differential equation of
type \cref{eq:M} for all these functions:
\[xF_n'=\frac{n+\lambda-1}{\epsilon_2m}(F_{n-1}-F_n).\]

\smallskip
\paragraph{\em Families III and IV}
For family III, we find the mixed equations
\begin{align*}
F_n&={}_{1}F_{1}\!\left(\left.\begin{matrix}
\mu+n\\
\lambda+n
\end{matrix}\right|x\right):&\quad F_n'&=\frac{n+\mu}{n+\lambda}F_n;\\
F_n&={}_{2}F_{1}\!\left(\left.\begin{matrix}
\mu+n,a\\
\lambda+n
\end{matrix}\right|x\right):&\quad (1-x)F_n'&=\frac{(n+\mu)
(n+\lambda-a)}{n+\lambda}F_{n+1}+(n+\mu)F_n.
\end{align*}
Changing $n$ into~$-n$ in these two relations gives the
analogues
for type IV.

\smallskip
\paragraph{\em Family V}
We obtain mixed equations in two cases:
\begin{align*}
&F_n={}_{2}F_{0}\!\left(\left.\begin{matrix}
1-\lambda-n,\mu+n\\
-
\end{matrix}\right|x\right):\\
&-x^2F_n'=
\frac{(n+\mu)(n+\lambda-1)}{(2n+\lambda+\mu)
(2n+\lambda+\mu-1)}F_{n+1}-{2\frac{(n+\mu)(n+\lambda-1)}{(2n+\lambda+\mu)
(2n+\lambda+\mu-2)}F_n}\\
&\qquad\qquad
+\frac{(n+\mu)(n+\lambda-1)}{
(2n+\lambda+\mu-1)(2n+\lambda+\mu-2)}F_{n-1};\\
&F_n={}_{2}F_{1}\!\left(\left.\begin{matrix}
1-\lambda-n,
\mu+n\\
b
\end{matrix}\right|x\right):\\
&\quad x(1-x)F_n'=
\frac{(n+\mu)(n+\lambda-1)(n+\lambda+b-1)}{(2n+\lambda+\mu)
(2n+\lambda+\mu-1)}F_{n+1}\\
&\quad-{\frac{(n+\mu)(n+\lambda-1)(\lambda-\mu+2b-2)}{(2n+\lambda+\mu)
(2n+\lambda+\mu-2)}F_n}
-\frac{(n+\mu)(n+\lambda-1)(n+\mu-b)}{
(2n+\lambda+\mu-1)(2n+\lambda+\mu-2)}F_{n-1};
\end{align*}

\smallskip
\paragraph{\em Family VI}
There, more mixed relations can be found:
\begin{align*}
&F_n={}_{0}F_{2}\!\left(\left.\begin{matrix}
-\\
\lambda+n,1-\mu-n
\end{matrix}\right|x\right):\\
&\quad -F_n'=
\frac{(n+\mu)}{(n+\lambda)(2n+\lambda+\mu)
(2n+\lambda+\mu-1)}F_{n+1}
+{\frac{2}{(2n+\lambda+\mu)
(2n+\lambda+\mu-2)}}F_n\\
&\qquad
+\frac{(n+\lambda-1)}{
(n+\mu-1)(2n+\lambda+\mu-1)(2n+\lambda+\mu-2)}F_{n-1};\\
\end{align*}
\begin{align*}
&F_n={}_{1}F_{2}\!\left(\left.\begin{matrix}
a_1\\
\lambda+n,1-\mu-n
\end{matrix}\right|x\right):\\
&\quad F_n'=
\frac{(n+\mu)(n+\lambda-a_1)}{(n+\lambda)(2n+\lambda+\mu)
(2n+\lambda+\mu-1)}F_{n+1}
+{\frac{\lambda-\mu-2a_1}{(2n+\lambda+\mu)
(2n+\lambda+\mu-2)}}F_n\\
&\qquad
-\frac{(n+\lambda-1)(n+\mu+a_1-1)}{
(n+\mu-1)(2n+\lambda+\mu-1)(2n+\lambda+\mu-2)}F_{n-1};\\
&F_n={}_{2}F_{2}\!\left(\left.\begin{matrix}
a_1,a_2\\
\lambda+n,1-\mu-n
\end{matrix}\right|x\right):\\
&\quad F_n'=
-\frac{(n+\mu)(n+\lambda-a_1)(n+\lambda-a_2)}{(n+\lambda)
(2n+\lambda+\mu)
(2n+\lambda+\mu-1)}F_{n+1}\\
&\qquad
+\left({\frac12+\frac{(\lambda-\mu-2a_1)(\lambda-\mu-2a_2)}
{4}\left(\frac1{2n+\lambda+\mu}-\frac1
{2n+\lambda+\mu-2}\right)}\right)F_n\\
&\qquad
-\frac{(n+\lambda-1)(n+\mu+a_1-1)(n+\mu+a_2-1)}{
(n+\mu-1)(2n+\lambda+\mu-1)(2n+\lambda+\mu-2)}F_{n-1};\\
&F_n={}_{3}F_{2}\!\left(\left.\begin{matrix}
a_1,a_2,a_3\\
\lambda+n,1-\mu-n
\end{matrix}\right|x\right):\\
&\quad (1-x)F_n'=
-\frac{(n+\mu)(n+\lambda-a_1)(n+\lambda-a_2)(n+\lambda-a_3)}{
(n+\lambda)
(2n+\lambda+\mu)
(2n+\lambda+\mu-1)}F_{n+1}\\
&\qquad
+\left(\frac{2(a_1+a_2+a_3)+\mu-\lambda-2}4\right.\\
&\qquad\qquad-\frac{
(\lambda-\mu-2a_1)
(\lambda-\mu-2a_2)(\lambda-\mu-2a_3)}
{8}\times\\
&\qquad\qquad\qquad\qquad\left.\left(\frac1{2n+\lambda+\mu}-\frac1
{2n+\lambda+\mu-2}\right)\right)F_n\\
&\qquad
-\frac{(n+\lambda-1)(n+\mu+a_1-1)(n+\mu+a_2-1)(n+\mu+a_3-1)}{
(n+\mu-1)(2n+\lambda+\mu-1)(2n+\lambda+\mu-2)}F_{n-1}.
\end{align*}

\bibliographystyle{abbrv}

\end{document}